\documentclass{amsart}

\usepackage{hyperref}
\usepackage{amssymb,amsmath,amsthm,graphicx,multirow,setspace,url, enumerate, float,tikz}
\usepackage{bbm} % for blackboard bold 1 type \mathbbm{1}
\usepackage{todonotes}
\usepackage[pagewise]{lineno}%\linenumbers

\numberwithin{equation}{section}

\newtheorem{theorem}{Theorem}[section]
\newtheorem{lemma}[theorem]{Lemma}
\newtheorem{corollary}[theorem]{Corollary}
\newtheorem{proposition}[theorem]{Proposition}

\theoremstyle{definition}
\newtheorem{definition}[theorem]{Definition}
\newtheorem{remark}[theorem]{Remark}
\newtheorem{example}[theorem]{Example}

\newtheorem*{theorem*}{Theorem}
\newtheorem*{corollary*}{Corollary}

\newtheorem*{intro:rank1}{Theorem \ref{rank1}}
\newtheorem*{intro:bridge}{Theorem \ref{T:generalbridge}}

\def\Z{\mathbb{Z}}

\def\a{\alpha}

\begin{document}
	
%%%%%%%%%%%%%%%%%%%%%%%%%%%%%%%%%%%%%%%%%%%%%%%%%%%%%%%%%%%%%%%%%%%%%%
%                             ABSTRACT
%\begin{abstract}
%	blabityblah
%\end{abstract}

%%%%%%%%%%%%%%%%%%%%%%%%%%%%%%%%%%%%%%%%%%%%%%%%%%%%%%%%%%%%%%%%%%%%%
%             TITLE - AUTHORS - DATE - SUBJECT AND KEYWORDS
 
\title{Graphs admitting only constant splines}

\author[K. Anders]{Katie Anders}
\address{University of Texas at Tyler}
\email{kanders@uttyler.edu}

\author[A. Crans]{Alissa S. Crans}
\address{Loyola Marymount University}
\email{acrans@lmu.edu}

\author[B. Foster-Greenwood]{Briana Foster-Greenwood}
\address{California State Polytechnic University, Pomona}
\email{brianaf@cpp.edu}

\author[B. Mellor]{Blake Mellor}
\address{Loyola Marymount University}
\email{blake.mellor@lmu.edu}

\author[J. Tymoczko]{Julianna Tymoczko}
\address{Smith College}
\email{jtymoczko@smith.edu}

\begin{abstract}
We study {\em generalized graph splines,} introduced by Gilbert, Viel, and the last author \cite{gpt}.  For a large class of rings, we characterize the graphs that only admit constant splines.  To do this, we prove that if a graph has a particular type of cutset (e.g., a bridge), then the space of splines naturally decomposes as a certain direct sum of submodules. As an application, we use these results to describe splines on a triangulation studied by Zhou and Lai, but over a different ring than they used.
\end{abstract}

\maketitle

%%%%%%%%%%%%%%%%%%%%%%%%%%%%%%%%%%%%%%%%%%%%%%%%%%%%%%%%%%%%%%%%%%%%%%%%%%%%%%%%
%%%%%%%%%%%%%%%%%%%%%%%%%%%%%%%%%%%%%%%%%%%%%%%%%%%%%%%%%%%%%%%%%%%%%%%%%%%%%%%%
%                         SECTION: INTRODUCTION
%%%%%%%%%%%%%%%%%%%%%%%%%%%%%%%%%%%%%%%%%%%%%%%%%%%%%%%%%%%%%%%%%%%%%%%%%%%%%%%%
%%%%%%%%%%%%%%%%%%%%%%%%%%%%%%%%%%%%%%%%%%%%%%%%%%%%%%%%%%%%%%%%%%%%%%%%%%%%%%%%
\section{Introduction}

%In 1946, mathematicians appropriated the term {\it spline} from shipbuilding, where engineers modeled the hull of a ship by anchoring thin strips of wood, known as splines, to select points on the hull and allowed the strips to naturally flex to approximate the rest of the shape of the hull.  

This paper studies generalized splines, which are parametrized by a ring, a graph, and a map from the edges of the graph to ideals in the ring. As the name suggests, they generalize the classical splines from analysis and applied mathematics.  The main goal of this paper is to describe when the module of generalized splines has rank one over the base ring.  We give multiple equivalent conditions for this to be true over a large family of rings.  Our main tools are of independent interest: different reductions on the module of splines depending on either algebraic characteristics of the edge labeling or combinatorial characteristics of the graph.  

Classically, a spline is a collection of polynomials on the faces of a polyhedral complex that agree to specified degree of smoothness on the intersections of faces. More formally, given a simplicial complex $\Delta$ we define the vector space of {\em splines} $S_d^r(\Delta)$ to be the space of all piecewise polynomial functions on $\Delta$ that have degree $d$ and order of smoothness $r$.  Splines are a standard topic in numerical analysis and are used in data interpolation, geometric design, and to approximate solutions to partial differential equations, among other applications.  Splines are also studied from a more theoretical perspective by analysts (see \cite{SchLai07} for a survey).  Two fundamental problems in both contexts are to find a basis (often satisfying specified constraints) or to compute the dimension of the vector space of splines over a given polyhedral complex, e.g., \cite{strang, AlfSchu, aps, ChuiLai, SchLai07}.

Our approach is essentially dual to that of classical splines.  %Billera pioneered the use of sophisticated algebraic and homological techniques to analyze splines \cite{Billera88}.  
Billera and Rose observed that splines can be viewed as functions on the {\em dual graph} of the polyhedral complex, reinterpreting the order-of-smoothness condition as a compatibility condition across each edge of the graph \cite{br}.  Independently, splines were reinvented in a combinatorial construction of equivariant cohomology, often called {\em GKM theory} by symplectic geometers and algebraic topologists \cite{gkm, Tym16}, and also arise naturally in the context of toric varieties in algebraic geometry \cite{Pay06}. 

The following definition unifies and generalizes previous work, allowing us to define splines on an arbitrary graph (not just those that can be realized in geometric settings) and to consider a much larger class of rings (not just polynomial rings). It first appeared in work of Gilbert, Viel, and the last author of this work \cite{gpt}, and extends earlier work by Guillemin and Zara to put GKM theory in a combinatorial context (see \cite{GuiZar00}, also e.g., \cite{GuiZar01b, GuiZar01a, GuiZar03}).

\begin{definition}     
Given a graph $G = (V, E)$ and a commutative ring $R$ with unit, an \textit{edge labeling} of $G$ is a function $\alpha: E \rightarrow I(R)$, where $I(R)$ is the set of ideals of $R$.  A \textit{spline} $p$ on $(G,\alpha)$ is a vertex labeling $p:V\rightarrow R$ such that for each edge $uv$, the difference $p(u)-p(v)$ lies in the ideal $\alpha(uv)$.  Let $S_R(G,\alpha)$ denote the set of all splines on $G$ with labels from $R$ and edge labeling $\alpha$.   When $R$ and $\alpha$ are clear from context, we write $S(G)$.
\end{definition}

In geometric applications, we have an underlying complex algebraic variety with a well behaved torus action.  The ring $R$ is the collection of polynomials in $n$ variables, where $n$ is the rank of the torus.  The graph corresponds to the 1-skeleton of the moment polytope with respect to the torus action and the labeling $\alpha$ records the weight of the torus action on each 1-dimensional torus orbit in the variety.

Note that $S_R(G, \alpha)$ is both a ring (with addition and multiplication of splines defined pointwise) and an $R$-module~\cite{gpt}. We focus on the $R$-module structure of $S_R(G,\alpha)$ in this paper.  Any collection of $R$-module generators of $S_R(G,\alpha)$ generates $S_R(G,\alpha)$ as a ring.  However, a minimal set of ring generators can be much smaller, and computing a multiplication table is generally very difficult--in fact, it is a longstanding open problem to compute the multiplication table with respect to the Schubert basis for the ring of splines corresponding to the equivariant cohomology of the full flag variety.

Analogous to the classical problem of finding a basis for the vector space of splines, a core problem in the theory of generalized graph splines is to describe minimal generating sets for $S_R(G, \alpha)$ as an $R$-module.  This is not difficult to do when the graph is a tree \cite{gpt} but is significantly harder for graphs containing cycles.  Various authors have studied the existence, size and construction of these generating sets for cycles and other graphs over the integers, the rings $\Z_m$, and other rings \cite{AltSar2, AltSar1, bhkr, bt, dip, GuiZar01b, GuiZar03, hmr, pstw}.

When the ring $R$ is an integral domain, then every generating set for $S_R(G, \alpha)$ has at least $n$ elements, where $n$ is the number of vertices in the graph $G$ \cite{gpt}. If $R$ is not an integral domain, however, this is no longer true; in fact, there are nontrivial labeled graphs that admit only constant splines \cite{bt}.  Determining whether there are nonconstant splines is thus a fundamental question in the field.

Our primary goal in this paper is to describe the edge labeled graphs which only admit constant splines. In Theorem \ref{rank1} we provide two equivalent combinatorial characterizations of these graphs for a large family of rings, one in terms of cutsets and the other in terms of spanning trees.

\begin{intro:rank1}%[\ref{rank1}] 
Let $R=\displaystyle\bigoplus_{i=1}^k R_i$, where each $R_i$ is an irreducible commutative ring with identity.  Let $G$ be a connected graph with edge set $E(G)$ and edge labeling $\alpha$.  Then the following are equivalent.
\begin{enumerate}
\item The edge labeled graph $(G,\alpha)$ has rank one.
\item For any cutset $C\subset E(G)$, the intersection $\displaystyle\bigcap_{e\in C}\alpha(e)=0$.
\item The graph $G$ has spanning trees $T_1, \ldots, T_k$ such that all edge labels of $T_i$ are contained in $\displaystyle\bigoplus_{j\neq i}^k R_j$ for all  $1\leq i\leq k$.
\end{enumerate}
\end{intro:rank1}

Along the way, we study operations that induce a decomposition of the module of splines. We take two approaches, with the following main results:
\begin{itemize}
\item In Corollary \ref{reducetotree} we give an algebraic condition on the labelling $\alpha$ under which the splines $S_R(G,\alpha)$ can be reduced (as a ring and as an $R$-module) to the splines on a spanning tree of $G$.  This condition is essentially that the edge labels form a set of nested ideals.
\item In Theorem \ref{T:generalbridge} we give a combinatorial condition on the graph $G$ under which the $R$-module $S_R(G,\alpha)$ decomposes into a direct sum of specific submodules.  Our condition generalizes the graph-theoretic notion of a bridge and applies to {\em any} choice of ring $R$.
\end{itemize}

%\begin{intro:bridge}%[\ref{T:generalbridge}]
%	Suppose $(G,\alpha)$ is an edge labeled graph with a subgraph $H$ such that each connected component of the graph $G-E(H)$ contains exactly one vertex of $H$. Suppose $H$ has vertices $h_1, \dots, h_n$, and $G_i$ is the component of $G-E(H)$ containing $h_i$.  For every graph $G$ and vertex $v$ in $G$, let $S(G; v)$ denote the splines on $G$ that are zero at $v$.  Then $S(G)$ is module-isomorphic to the direct sum
%$$S(G) = \langle \mathbbm{1}_G \rangle \oplus {S}(H; h_1) \oplus {S}(G_1; h_1) \oplus \cdots \oplus {S}(G_n; h_n)$$
%\end{intro:bridge}

As an application, we consider splines on a triangulation studied by Zhou and Lai \cite{zl} but over integers mod $m$ rather than their polynomial rings. The contrast between our results and theirs demonstrates the significant differences that can occur when changing the base ring.

%%%%%%%%%%%%%%%%%%%%%%%%%%%%%%%%%%%%%%%%%%%%%%%%%%%%%%%%%%%%%%%%%%%%%%%%%%%%%%%%
%%%%%%%%%%%%%%%%%%%%%%%%%%%%%%%%%%%%%%%%%%%%%%%%%%%%%%%%%%%%%%%%%%%%%%%%%%%%%%%%
%                         SECTION: DECOMPOSING GRAPHS
%%%%%%%%%%%%%%%%%%%%%%%%%%%%%%%%%%%%%%%%%%%%%%%%%%%%%%%%%%%%%%%%%%%%%%%%%%%%%%%%
%%%%%%%%%%%%%%%%%%%%%%%%%%%%%%%%%%%%%%%%%%%%%%%%%%%%%%%%%%%%%%%%%%%%%%%%%%%%%%%

\section{Decomposing Graphs} \label{decompose}

In this section we decompose splines based on different combinatorial conditions of the graph or algebraic conditions on the labeling.  In what follows, we assume that our graphs $G$ are connected because the ring (and module) of splines naturally decomposes over disconnected components. 

The following lemma describes operations on graphs that allow us to assume no edge is labeled either by the ideal $(0)$ or by the ideal $(1)$.  It was proven by Gilbert, the final author, and Viel \cite{gpt}. 

\begin{lemma}[Gilbert, Tymoczko, Viel] \label{contract zero}
 If $\alpha(uv)$ is the ideal generated by $0$ and $G/uv$ is the graph obtained by contracting edge $uv$ then $S(G) \cong S(G/uv)$. If $\alpha(uv)$ is the whole ring (generated by $(1)$ in $R$), then  $S(G)\cong S(G-uv)$, where $G-uv$ is the graph obtained by deleting the edge $uv$.
\end{lemma}

Contractions can produce loops or multiple edges, so using Lemma~\ref{contract zero} appears to require splines on {\em multigraphs} (which can have loops and multiple edges between two vertices).  However, the next result is the key step to prove that for each multigraph, there is a simple graph with the same ring of splines.  Our definition for splines on a multigraph is that the spline condition must be satisfied for each edge individually. 

\begin{lemma} \label{multigraph}
Suppose $(G,\alpha)$ is a multigraph, with two edges $e_1$ and $e_2$ between vertices $u$ and $v$.  Suppose $G'$ is the graph obtained by replacing $e_1$ and $e_2$ with a single edge $e$.  Suppose $\alpha'$ is defined by setting $\alpha'(e) = \alpha(e_1) \cap \alpha(e_2)$ and by setting $\alpha'(e') = \alpha(e')$ for all other edges $e'$.  Then $S(G, \alpha) = S(G', \alpha')$.
\end{lemma}
\begin{proof}
First suppose $p \in S(G, \alpha)$.  Then $p(u) - p(v) \in \alpha(e_1)$ and $p(u) - p(v) \in \alpha(e_2)$.  Thus $p(u) - p(v) \in \alpha(e_1) \cap \alpha(e_2) = \alpha'(e)$ so $p \in S(G', \alpha')$.  

Conversely, if $p \in S(G', \alpha')$, then $p(u) - p(v) \in \alpha'(e) = \alpha(e_1) \cap \alpha(e_2)$.  Thus $p(u) - p(v)$ is in both ideals $\alpha(e_i)$ and so $p \in S(G, \alpha)$.  This proves the claim.
\end{proof}

The next corollary is immediate.

\begin{corollary}
Suppose $(G,\alpha)$ is a multigraph.  Let $G'$ be the graph obtained by erasing every loop (namely edge of the form $vv$ for some vertex $v$) from $G$, and by erasing all but one edge in the event of multiple edges between any two vertices $uv$.  Let $\alpha'$ be the edge labeling obtained by assigning to the edge $uv$ in $G'$ the edge label $\alpha(e_1) \cap \alpha(e_2) \cap \cdots \alpha(e_k)$, where $e_1$, $e_2$, $\ldots$, $e_k$ are all edges between $u$ and $v$ in $G$.  Then $S(G,\alpha) = S(G',\alpha')$.
\end{corollary}
\begin{proof}
Loops start and end at the same vertex so the spline condition is trivially satisfied on each loop.  Thus loops can be removed without changing the space of splines. Applying Lemma \ref{multigraph} repeatedly then gives the result.	
\end{proof}

If $R$ is a field, then its only ideals are $(0)$ and $(1)$.  Lemma~\ref{contract zero} then says that the ring $S(G)$ is isomorphic to the ring of splines on an isolated set of, say, $n$ vertices, which is simply $R^n$.  For this reason, we generally do not consider splines over fields.

\subsection{Edges of cycles.} We give a new condition under which an edge of a cycle may be deleted without losing any information: that the ideal associated to one edge $uv$ of the cycle contains all of the other edge labels of the cycle.  Informally, in this situation the edge $uv$ contains information that is redundant with the rest of the cycle and so may be erased.  The next lemma states this formally; in the rest of the section, we elaborate some consequences.

\begin{lemma}\label{delete gcd}
	Let $(G,\alpha)$ be an edge labeled graph. Let $uv$ be an edge and suppose there is a cycle $C$ in $G$ that contains $uv$ such that the ideal $\alpha(uv)$ contains every edge label from $C$.  Then $S(G)=S(G-\{uv\})$.
\end{lemma}	
\begin{proof}
	Every spline on $G$ is also a spline on $G-\{uv\}$, so $S(G) \subseteq S(G-\{uv\})$.  
	
	Now suppose that $p \in S(G-\{uv\})$. Let $w_0w_1w_2\cdots w_{n}w_{n+1}w_0$ be a cycle in $G$ that contains the edge $uv$, say with $w_0=u$ and $w_{n+1}=v$.  Since
	\[
	p(u) - p(v) = p(u) - p(w_1) + p(w_1) - p(w_2) + \cdots + p(w_n) - p(v)
	\]
	and since $p(w_i) - p(w_{i+1}) \in \alpha(w_iw_{i+1})$ for each $i$, we know that
	\[p(u) - p(v) \in \sum_{i=0}^{n}{\alpha(w_iw_{i+1})}.\]
This is contained in $\alpha(uv)$ by the hypothesis that $\alpha(uv)$ contains all the edge labels $\alpha(w_iw_{i+1})$ for $0\leq i\leq n$.  	Hence $p$ is also a spline for $G$ and so $S(G-\{uv\}) \subseteq S(G)$.  The claim follows.
\end{proof}

Lemma~\ref{delete gcd} has particularly interesting consequences when the set of edge labels is linearly ordered (namely, when any two distinct edge labels $\alpha_1$ and $\alpha_2$ satisfy either $\alpha_1\subset\alpha_2$ or $\alpha_2\subset\alpha_1$).  In this case, we can reduce to considering the $R$-module of splines on a tree. 

\begin{corollary}\label{reducetotree}
	Suppose $(G,\alpha)$ is a connected graph whose set of edge labels can be linearly ordered by inclusion. Then there exists a spanning tree $T$ such that $S(G)=S(T)$.
\end{corollary}

\begin{proof}
	If $G$ is a tree, then we are done. Otherwise, the graph $G$ must contain a cycle. Since the set of edge labels can be linearly ordered, there must be an edge, say $uv$, of the cycle such that $\alpha(uv)$ contains all other ideals labeling the edges of that cycle. By
	Lemma~\ref{delete gcd}, we have $S(G)=S(G-\{uv\})$. Note that $G-\{uv\}$ is still connected, so either $G-\{uv\}$ is a
	tree or we can find another cycle and find an edge in that cycle to delete. Continuing until no cycles remain, we obtain a tree $T$ with $S(G)=S(T)$.
\end{proof}

In particular Corollary~\ref{reducetotree} applies when working over rings for which the set of {\em all} ideals is linearly ordered; such rings are known as {\it uniserial rings}. Examples include $\Z_{p^k}$ for $p$ prime and $F[x]/(x^n)$ where $F$ is a field. Furthermore, commutative uniserial rings that are domains are exactly valuation rings. 

\begin{remark}
Corollary \ref{reducetotree} and Corollary \ref{bridge} can together be used to describe bases of splines over $\Z_{p^e}$ for any prime $p$ (see Remark \ref{remark: recurse for trees} for more detail on how to generate bases of trees).  However, our results do not give an explicit construction like that in Philbin-Swift-Tammaro-Williams \cite{pstw}.
\end{remark}

%\begin{lemma}\label{delete gcd}
%	If $\alpha(uv)$ contains every edge label in $(G,\alpha)$, then $S(G)=S(G-\{uv\})$, provided $G-\{uv\}$ is connected. In particular, this is true if $\alpha(uv) = (1) = R$ (and, in this case, we can drop the requirement that $G-\{uv\}$ is connected).
%\end{lemma}	
%\begin{proof}
%	Clearly, every spline on $G$ is also a spline on $G-\{uv\}$, so $S(G) \subseteq S(G-\{uv\})$.  Now suppose that $p \in S(G-\{uv\})$.  Since $G-\{uv\}$ is connected, there is a path $uw_1w_2\dots w_nv$ from $u$ to $v$ in $G-\{uv\}$.  Then
%	\[
%	p(u) - p(v) = p(u) - p(w_1) + p(w_1) - p(w_2) + \cdots + p(w_n) - p(v).
%	\]
%	Since $p(w_i) - p(w_{i+1}) \in \alpha(w_iw_{i+1})$, this means $p(u) - p(v)$ is a sum of elements from the ideals $\alpha(w_iw_{i+1})$ (here we let $u = w_0$ and $v = w_{n+1}$).  Then $p(u) - p(v) \in \sum_{i=0}^{n}{\alpha(w_iw_{i+1})} \subseteq \alpha(uv)$, since $\alpha(uv)$ contains every other edge label.  Hence $p$ is also a spline for $G$, so $S(G-\{uv\}) \subseteq S(G)$.  We conclude that $S(G)=S(G-\{uv\})$.
%\end{proof}

%%%%%%%%%%%%%%%%%%%%%%%%%%%%%%%%%%%%%%%%%%%%%%%%%%%%%%%%%%%%%%%%%%%%%%
%                             Bridge Lemma

\subsection{Bridges.}

We now turn to the case of graphs that contain a {\em bridge}, i.e., an edge whose deletion disconnects the graph.  We will show that the module of splines on the graph is (almost) the direct sum of the spline modules for the two components after the bridge is deleted.  We find it convenient to define splines {\em based at a vertex}.

\begin{definition}
	Given a vertex $v$ of a graph $G$, we say that a spline $p \in S_R(G, \alpha)$ is {\em based at $v$} if $p(v) = 0$.  We denote the submodule of splines based at $v$ by $S_R(G, \alpha; v)$ and call $v$ the {\em basepoint} of this submodule. We write $S(G; v)$ if $R$ and $\alpha$ are clear.
\end{definition}

Our terminology is new but the idea is not, e.g., \cite{GuiZar03, KnuTao03, gpt}.
	
Note that $S(G; v)$ may equal $S(G; w)$ for distinct vertices $v$ and $w$.  This happens precisely when $p(v) = p(w)$ for every spline $p \in S(G)$. 

The next lemma restates an earlier result using our terminology \cite[Theorem~2.12]{gpt}.  It says that $S(G; v)$ is almost the same as $S(G)$.  We use $\mathbbm{1}_G \in S(G)$ to denote the constant spline that takes the value $1$ on every vertex.

\begin{lemma}[Gilbert, Tymoczko, Viel]\label{basepoint}
For any vertex $v$ of $G$, there is an $R$-module decomposition
 \[S_R(G) \cong S_R(G; v) \oplus \langle \mathbbm{1}_G \rangle \cong S_R(G; v) \oplus R.\]
\end{lemma}

\begin{remark}
The submodule of splines based at vertex $v$ is closed under multiplication and so is an ideal in $S_R(G)$.  However, the decomposition in Lemma~\ref{basepoint} is not a ring isomorphism. 
\end{remark}

%%%%%%%%%%%%%%%%%%%%%%%%%%%%%%%%%%%%%%%%%%%%%%%%%%%%%%%%%%%%%%%%%%%%%%%%%%
% GENERALIZED BRIDGE LEMMA

Extending Lemma~\ref{basepoint}, we next show that if a graph has a bridge, its space of splines is (almost) the direct sum of the modules of splines on the two components created by removing the bridge.  In fact, we prove this in an even more general setting, one in which the ``bridge" is a subgraph $H$ that is more complicated than an edge.  We still require this ``generalized bridge" to disconnect the graph and to meet each component $G_i$ of the disconnected graph in a single vertex $h_i$, as in Figure \ref{F:generalbridge}.
\begin{figure}
	\definecolor{cqcqcq}{rgb}{0.7529411764705882,0.7529411764705882,0.7529411764705882}
	\definecolor{xfqqff}{rgb}{0.4980392156862745,0.,1.}
	\begin{tikzpicture}[scale=1.25]
	\clip(-0.74,0.22) rectangle (4.3,4.26);
	\draw [rotate around={5.454388281460479:(0.457830866504267,2.7651028948327334)},line width=1.2pt] (0.457830866504267,2.7651028948327334) ellipse (0.9653699809713796cm and 0.7215153091054788cm);
	\draw [rotate around={-37.10152679123632:(3.038405425016262,1.5869133474406685)},line width=1.2pt] (3.038405425016262,1.5869133474406685) ellipse (0.8914762430709756cm and 0.5107956850618345cm);
	\draw [rotate around={0.8658916839396917:(1.0230471280239055,1.1146398891866771)},line width=1.2pt] (1.0230471280239055,1.1146398891866771) ellipse (0.9532578054674283cm and 0.5420282616525114cm);
	\draw [rotate around={-33.86090315007394:(2.846455190913795,3.1858829712756807)},line width=1.2pt] (2.846455190913795,3.1858829712756807) ellipse (0.840960086417234cm and 0.5795935280113562cm);
	\draw [line width=1.2pt,dotted,color=cqcqcq] (1.02,2.48)-- (2.66,2.86);
	\draw [line width=1.2pt,dotted,color=cqcqcq] (2.66,2.86)-- (2.66,1.94);
	\draw [line width=1.2pt,dotted,color=cqcqcq] (2.66,1.94)-- (1.5,1.24);
	\draw [line width=1.2pt,dotted,color=cqcqcq] (1.5,1.24)-- (1.02,2.48);
	\draw [line width=1.2pt,dotted,color=cqcqcq] (1.02,2.48)-- (2.66,1.94);
	\draw [line width=1.2pt,dotted,color=cqcqcq] (1.5,1.24)-- (2.66,2.86);
	\begin{scriptsize}
	\draw[color=black] (-0.23,3.06) node {$G_1$};
	\draw[color=black] (3.37,1.46) node {$G_3$};
	\draw[color=black] (0.31,1.22) node {$G_4$};
	\draw[color=black] (2.47,3.76) node {$G_2$};
	\draw [fill=xfqqff] (1.02,2.48) circle (1.5pt);
	\draw[color=xfqqff] (1.05,2.98) node {$h_1$};
	\draw [fill=xfqqff] (2.66,2.86) circle (1.5pt);
	\draw[color=xfqqff] (2.91,3.16) node {$h_2$};
	\draw [fill=xfqqff] (2.66,1.94) circle (1.5pt);
	\draw[color=xfqqff] (2.93,2.1) node {$h_3$};
	\draw [fill=xfqqff] (1.5,1.24) circle (1.5pt);
	\draw[color=xfqqff] (1.35,1.16) node {$h_4$};
	\end{scriptsize}
	\end{tikzpicture}
	\caption{A generalized bridge.}
	\label{F:generalbridge}
\end{figure}
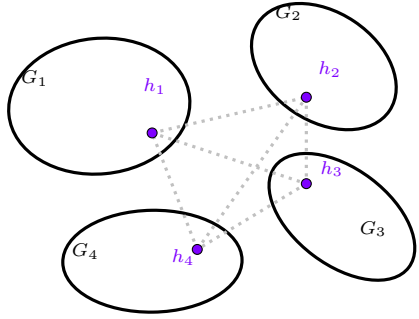
Intuitively, our proof decomposes $S(G)$ into a direct sum of splines on $H$ and on each $G_i$.  More formally, we have the following.  

\begin{theorem} \label{T:generalbridge}
	Suppose $(G,\alpha)$ is an edge labeled graph with a subgraph $H$ such that each connected component of the graph $G-E(H)$ contains exactly one vertex of $H$. Suppose $H$ has vertices $h_1, \dots, h_n$ and $G_i$ is the component of $G-E(H)$ containing $h_i$.  Then $S(G)$ is isomorphic to the direct sum of $R$-modules:
$$S(G) \cong \langle \mathbbm{1}_G \rangle \oplus {S}(H; h_1) \oplus {S}(G_1; h_1) \oplus \cdots \oplus {S}(G_n; h_n).$$
Moreover, for each $i$, let $\widetilde{S}(G_i; h_i)$ denote the subset of splines in $S(G)$ that are zero when restricted to the vertices in $H \cup \bigcup_{j \neq i} G_j$.  Let $\widetilde{S}(H)$ denote the splines in $S(G)$ such that for each $i$ the restriction to $G_i$ is constant. (If $i \neq j$ then the restrictions to $G_i$ and $G_j$ need not agree.) Then $S(G)$ is the internal direct sum of submodules 
$$S(G) = \langle \mathbbm{1}_G \rangle \oplus \widetilde{S}(H; h_1) \oplus \widetilde{S}(G_1; h_1) \oplus \cdots \oplus \widetilde{S}(G_n; h_n).$$
\end{theorem}

\begin{proof} 
Suppose $p$ is a spline in $S(G)$.  The spline defined by $p' = p - p(h_1)\mathbbm{1}_G$ is in $S(G; h_1)$.  Let $p'_H$ be the restriction of $p'$ to $H$ and let $\widetilde{p'_H}$ be the extension of $p'_H$ to $G$ defined by setting $\widetilde{p'_H}(u)=p'(h_i)$ for all $i$ and vertices $u$ in $G_i$.  Note that $\widetilde{p'_H} \in \widetilde{S}(H; h_1) \subseteq S(G; h_1)$.  Now for each $i$, let $q_i$ be the restriction of $p' - \widetilde{p'_H}$ to $G_i$, so $q_i \in S(G_i; h_i)$.  Extend $q_i$ to the spline $\widetilde{q_i} \in S(G)$ by setting $\widetilde{q_i}(u) = 0$ for all vertices $u$ off of $G_i$.  Note that $\widetilde{q_i} \in \widetilde{S}(G_i; h_i)$.  Thus by construction
$$p = p(h_1)\mathbbm{1}_G + \widetilde{p'_H} + \sum_{i=1}^n{\widetilde{q_i}} \in \langle \mathbbm{1}_G \rangle \oplus \widetilde{S}(H; h_1) \oplus \widetilde{S}(G_1; h_1) \oplus \cdots \oplus \widetilde{S}(G_n; h_n).$$
	
Now we show that this decomposition of $p$ is unique.  Suppose 
$$p = r\mathbbm{1}_G + q_0 + q_1 + q_2 + \cdots + q_n  = 0,$$
where $q_0 \in \widetilde{S}(H; h_1)$ and $q_i \in \widetilde{S}(G_i,h_i)$ for $1 \leq i \leq n$.  Note that $p(h_i) = q_0(h_i) + r$ for each $i$ and, in particular, $p(h_1) = r$.  Since $p = 0$ this means $r = 0$.  Then 
$$\widetilde{p'_H} = q_0 + r\mathbbm{1}_G = q_0$$
since each $q_i$ is zero on $H$ and $p$ is constant on each $G_i$.  But since $p = 0$, we can explicitly compute $p'_H = 0$ and so $q_0 = 0$.  An analogous argument shows that for each $G_i$, the spline $$\widetilde{p}_{G_i} = q_i + (q_0(h_i) + r)\mathbbm{1}_G = q_i = 0,$$
and so the decomposition is unique.  

Finally, note that the restriction map is a natural module isomorphism between the submodule $\widetilde{S}(H; h_1)$ and $S(H; h_1)$, respectively $\widetilde{S}(G_i; h_i)$ and $S(G_i; h_i)$.  This proves the claim.
\end{proof}

We apply this decomposition to the special case of a graph with a bridge that is labeled by a principal ideal; it generalizes to any graph with a bridge if the single generator $\beta \chi_B$ is replaced by the set of generators of the ideal labeling the bridge.

\begin{corollary}\label{bridge}
Let $G$ be a connected graph with bridge $ab$ and edge labeling $\alpha$ such that the bridge is labeled by the principal ideal $\alpha(ab) = (\beta)$.  The graph $G-\{ab\}$ has two components; denote the component that contains $a$ by $A$ and the component that contains $b$ by $B$. Let $\chi_B$ represent the function ({\em not} necessarily a spline) that is identically 1 on the vertices of $B$ and 0 on the vertices of $A$.  Then we have an $R$-module decomposition
$$S(G) \cong S(A; a) \oplus S(B; b) \oplus \langle \mathbbm{1}_G \rangle \oplus \langle \beta \chi_B \rangle.$$
\end{corollary}
\begin{proof}
This follows directly from the previous theorem once we note that the graph consisting simply of the edge $ab$ is generated as an $R$-module by the identity spline together with the spline that is $\beta$ on $b$ and $0$ on $a$ and that $\chi_B$ is the extension of this latter spline to all of $G$.
\end{proof}

%As a consequence of Lemma~\ref{delete gcd} and Corollary~\ref{bridge}, we can easily compute the space of splines for rings $R$ that have the property that all the ideals are in a nested hierarchy (i.e., for any two ideals $\alpha$ and $\beta$, either $\alpha \subset \beta$ or $\beta \subset \alpha$); these are known as {\em uniserial rings} (the rings $\mathbb{Z}_{p^k}$ for $p$ prime and $F[x]/(x^n)$, where $F$ is a field, are such rings).  In fact, we need only that the edge labels satisfy this nested relationship (rather than all ideals in the ring).  In this case, on any finite graph $(G, \alpha)$, there is an edge with a maximal label (i.e., the ideal labeling that edge contains every other edge label), and we can delete this edge by Lemma \ref{delete gcd} unless it is a bridge.  If it is a bridge, then we can delete it using Corollary~\ref{bridge}.  Continuing in this way (and keeping track of the labels on any bridges deleted), we can reduce the graph to a collection of isolated vertices and obtain a basis for the space of splines. 

\begin{remark} \label{remark: recurse for trees}
If $G$ is a tree, we can recursively apply Corollary ~\ref{bridge} to obtain a minimal generating set for $S(G)$ as an $R$-module. In this sense, our result generalizes the construction of bases of splines for trees in \cite{gpt}.
\end{remark}

The following example illustrates how Lemma~\ref{delete gcd}, Corollary~\ref{reducetotree}, and Corollary~\ref{bridge} can be used to find a minimal set of generators for the $R$-module of splines on a graph whose edge labels can be linearly ordered by inclusion. 

\begin{example}
Let $\beta, \gamma, \delta\in R$.  Suppose we have the graph $G$ with vertices $a,b,c$ and edge labeling $\alpha$ given below.  Let $G' = G - \{bc\}$ and $G'' = G' - \{ab\}$.

\begin{center}
\scalebox{.9}{\begin{tikzpicture}
\node at (-1, 1.5) {$G =$};
\draw (0,0) -- (0,3) node[midway, right] {$(\delta\beta\gamma)$};
\draw (0,0) -- (2,1.5); %node[midway, right] {$\;\;(\gamma)$}%
\node at (1.25, .4) {$(\delta)$};
\draw (0,3) -- (2,1.5); %node[midway, right] {$\;\;(\alpha)$};
\node at (1.25, 2.6) {$(\delta\beta)$};
\draw [fill] (0,0) circle (0.15);
\draw [fill] (0,3) circle (0.15);
\draw [fill] (2,1.5) circle (0.15);
\node [right] at (2,1.5) {$\;b$};
\node [above left] at (0,3) {$a$};
\node [below left] at (0,0) {$c$};

\node at (4, 1.5) {$G' =$};
\draw (5,0) -- (5,3) node[midway, right] {$(\delta\beta\gamma)$};
%\draw (5,0) -- (7,1.5); %node[midway, right] {$\;\;(\gamma)$}%
%\node at (6.25, .4) {$(\delta)$};
\draw (5,3) -- (7,1.5); %node[midway, right] {$\;\;(\alpha)$};
\node at (6.25, 2.6) {$(\delta\beta)$};
\draw [fill] (5,0) circle (0.15);
\draw [fill] (5,3) circle (0.15);
\draw [fill] (7,1.5) circle (0.15);
\node [right] at (7,1.5) {$\;b$};
\node [above left] at (5,3) {$a$};
\node [below left] at (5,0) {$c$};

\node at (9, 1.5) {$G'' =$};
\draw (10,0) -- (10,3) node[midway, right] {$(\delta\beta\gamma)$};
%\draw (10,0) -- (12,1.5); %node[midway, right] {$\;\;(\gamma)$}%
%\node at (11.25, .4) {$(\delta)$};
%\draw (10,3) -- (12,1.5); %node[midway, right] {$\;\;(\alpha)$};
%\node at (11.25, 2.6) {$(\delta\beta)$};
\draw [fill] (10,0) circle (0.15);
\draw [fill] (10,3) circle (0.15);
\draw [fill] (12,1.5) circle (0.15);
\node [right] at (12,1.5) {$\;b$};
\node [above left] at (10,3) {$a$};
\node [below left] at (10,0) {$c$};

\end{tikzpicture}}
\end{center}

Since $\alpha(bc)=(\delta)$ contains every other edge label of the cycle, we have $S(G)=S(G')$ by Lemma \ref{delete gcd}. Thus we have reduced to a tree as described by Corollary ~\ref{reducetotree}.  To apply Corollary ~\ref{bridge}, we consider the bridge $ab$ and let $G'' = G' -  \{ab\}$.  Let $A$ be the component of $G''$ containing $a$ and $B$ be the component of $G''$ containing $b$.  If the coordinates are ordered $(a,b,c)$, then the generators of $S(A; a)$ are $\mathcal{B}_a=\{(0,0,\delta\beta\gamma)\}$ and $S(B; b) = 0$.  By Corollary~\ref{bridge} we have 
$$S(G) = S(G') = S(A; a) \oplus S(B; b) \oplus \langle \mathbbm{1}_G \rangle \oplus \langle \delta\beta \chi_B \rangle$$ 
for which a minimal generating set is $U=\{(0,0,\delta\beta\gamma), (1,1,1),(0,\delta\beta,0)\}$.
\end{example}

%%%%%%%%%%%%%%%%%%%%%%%%%%%%%%%%%%%%%%%%%%%%%%%%%%%%%%%%%%%%%%%%%%%%%%%%%%%%%%%%
%%%%%%%%%%%%%%%%%%%%%%%%%%%%%%%%%%%%%%%%%%%%%%%%%%%%%%%%%%%%%%%%%%%%%%%%%%%%%%%%
%                         SECTION: CONSTANT SPLINES
%%%%%%%%%%%%%%%%%%%%%%%%%%%%%%%%%%%%%%%%%%%%%%%%%%%%%%%%%%%%%%%%%%%%%%%%%%%%%%%%
%%%%%%%%%%%%%%%%%%%%%%%%%%%%%%%%%%%%%%%%%%%%%%%%%%%%%%%%%%%%%%%%%%%%%%%%%%%%%%%%

\section{Graphs with rank one}\label{sec:constantsplines}

In this section, we consider labeled graphs that admit only constant splines, which we call graphs with rank one.

\begin{definition}\label{def:constantsplines}
Given a graph $G$ with edge labeling $\alpha$, a \textit{constant spline} on $(G,\alpha)$ takes the same value on every vertex of $G$.  The graph $(G,\alpha)$ has \textit{rank one} over the ring $R$ if it admits only constant splines.
\end{definition}

Graphs with rank one are particularly interesting from the point of view of subgraphs.  By Theorem \ref{T:generalbridge}, if the ``bridge" graph $H$ has rank one, then its contribution to the space of splines is just zero.  More generally, if a subgraph has rank one, it can be contracted to a vertex without changing the space of splines.  Our main result is a characterization of rank one graphs over rings which are direct sums of irreducible rings.

\begin{definition} \label{irreducible}
Recall that an ideal $I$ of a ring $R$ is \textit{irreducible} (sometimes called \textit{meet-irreducible}) if it is not the intersection of two strictly larger ideals.  A ring $R$ is \textit{irreducible} if $(0)$ is an irreducible ideal in $R$.  In other words, a ring $R$ is irreducible if the intersection of any two non-zero ideals is non-zero.
\end{definition}

\begin{remark}\label{R:irreducible}
Examples of irreducible rings include integral domains, uniserial rings such as $\Z_{p^k}$ for $p$ prime, and Artinian Gorenstein rings. An irreducible ring cannot be decomposed as the direct sum of two non-trivial rings, though the converse is not true. For example, the ring $\mathbb{R}[x,y]/(x^2,y^2,xy)$ cannot be decomposed as a direct sum of two of its ideals because one of those two ideals must contain polynomials with a non-zero constant term, and $ax+by+c$ generates the entire ring whenever $c \neq 0$.  However, this quotient is not irreducible since $(x) \cap (y) = (0)$.
\end{remark}

We now consider splines over a ring $R=\bigoplus_{i=1}^k R_i$, where each $R_i$ is an irreducible commutative ring with identity $1_{R_i}$. Let $\pi_i:R\rightarrow R_i$ be the canonical projection homomorphism and let $M_i$ denote the kernel of $\pi_i$. 

\begin{lemma}\label{closure}
	Suppose $R=\bigoplus_{i=1}^k R_i$ and that each
	$R_i$ is an irreducible commutative ring with identity $1_{R_i}$.
	If $I$ and $J$ are ideals of $R$ that are not contained in $M_i$, then $I\cap J$ is not contained in $M_i$.
\end{lemma}
\begin{proof}
	Suppose that $I$ and $J$ are ideals of $R$ that are not contained in  $M_i$ for some $i$.  It follows that $\pi_i(I)\neq0$ and $\pi_i(J)\neq0$. Since $R_i$ is irreducible,
	the ideal $\pi_i(I)\cap\pi_i(J)$ contains a non-zero element, say $x$. There must exist elements $r\in I$ and $s\in J$ such that $\pi_i(r)=\pi_i(s)=x$. Multiplying $r$ and $s$ each by $(0,\ldots,0,1_{R_i},0,\ldots,0)$ produces an element
	\((0,\ldots,0,x,0,\ldots,0)\)
	which belongs to both $I$ and $J$ but has non-zero $i$-th coordinate. Hence $I\cap J$ is not contained in $M_i$.  This proves the claim.
\end{proof}

We use this lemma to prove our main theorem, which is a complete graph-theoretic characterization of the graphs of rank one over rings $R$ of this form.

\begin{theorem} \label{rank1}
Suppose $R=\bigoplus_{i=1}^k R_i$ and that each
	$R_i$ is an irreducible commutative ring with identity $1_{R_i}$.  Let $G$ be a connected graph with edge set $E(G)$ and edge labeling $\alpha$.  Then the following are equivalent.
\begin{enumerate}
\item The edge labeled graph $(G,\alpha)$ has rank one over $R$.
\item For any cutset $C\subset E(G)$, the intersection $\displaystyle\bigcap_{e\in C}\alpha(e)=0$.
\item The graph $G$ has spanning trees $T_1, \ldots, T_k$ such that all edge labels of $T_i$ are contained in $M_i$ for all  $1\leq i\leq k$.
\end{enumerate}
\end{theorem}

\begin{remark} \label{remark:irreducible} When $R$ itself is irreducible, equivalently $k=1$, the third condition reduces to $G$ having a spanning tree with edges labeled by the zero ideal.  In general, the spanning trees $T_i$ need not be disjoint (or even distinct).	\end{remark}

\begin{proof}
%We prove $(1)\Rightarrow(2)$ by contrapositive. Let $C\subset E(G)$ be a cutset of $G$.   Then there exist non-empty subgraphs $G_1$ and $G_2$ of $G$ such that if $u\in G_1$ and $v\in G_2$, then the edge $uv\in C$.  Suppose there exists an non-zero element $x$ of $\bigcap_{e\in C}\alpha(e)$. Define a vertex labeling $p:V\rightarrow R$ by $\rho(u)=x$ for any $u\in G_1$ and $\rho(v)=0$ for any $v\in G_2$.  The restriction of $p$ to $G_1$ is a spline, as is the restriction of $p$ to $G_2$.  For any $u\in G_1$ and any $v\in G_2$, we have that $p(u)-p(v)=x\in \bigcap_{e\in C}\alpha(e)$.  Since the edge $uv\in C$, we have $x\in\alpha(uv)$.  Thus $p$ is a non-constant spline on $G$, and $(G,\alpha)$ does not have rank one.

First we prove $(1)\Rightarrow(2)$ by contrapositive. Let $C\subset E(G)$ be a cutset of $G$.   Without loss of generality, we assume that $C$ is minimal in the sense that no proper subset of $C$ is a cutset. Then there exists  a partition of the vertex set $V$ into nonempty subsets $V_1$ and $V_2$ such that $C=\{uv\in E(G)\mid u\in V_1, v\in V_2\}$.  Suppose there exists a non-zero element $x$ of $\bigcap_{e\in C}\alpha(e)$. Define a vertex labeling $p:V\rightarrow R$ by $p(u)=x$ for all $u\in V_1$ and $p(v)=0$ for all $v\in V_2$.  The restriction of $p$ to $V_1$ is a spline, as is the restriction of $p$ to $V_2$.  For any $u\in V_1$ and any $v\in V_2$, we have  $p(u)-p(v)=x\in \bigcap_{e\in C}\alpha(e)$.  The edge $uv$ is in $C$, so we conclude $x\in\alpha(uv)$.  Thus $p$ is a nonconstant spline on $G$, so $(G,\alpha)$ does not have rank one.

Next we prove $(2) \Rightarrow (3)$ by contrapositive.
For $1\leq i\leq k$, consider the equivalence relation $\equiv_i$ on $V(G)$ defined by $u\equiv_i v$ if and only if there exists a path from $u$ to $v$ having all edge labels contained in $M_i$.  Now suppose $i$ is an index such that there is no spanning tree $T_i$ such that all of its edge labels are contained in $M_i$. Then the relation $\equiv_i$ determines more than one equivalence class. Let $V_1$ be one of the equivalence classes and let $V_2$ be the union of the rest.
Note that if $v_1v_2$ is an edge with $v_1\in V_1$ and $v_2\in V_2$, then $\alpha(v_1v_2)$ cannot be contained in $M_i$ since $v_2\not\equiv_i v_1$. Now $C:=\{v_1v_2\in E(G)\mid v_1\in V_1, v_2\in V_2\}$ is a cutset (which is nonempty since $G$ is connected).  Lemma~\ref{closure} implies
that $\bigcap_{e\in C}\alpha(e)\not\subseteq M_i$. Hence $C$ is a cutset for which the intersection $\bigcap_{e\in C}\alpha(e)$ is non-zero.

Finally, we prove $(3)\Rightarrow (1)$. Suppose that for all $1\leq i\leq k$ the graph $(G,\alpha)$ has a spanning tree $T_i$ whose edge labels are all in $M_i$. Suppose $p$ is a spline on $G$. Since every pair of vertices $u$ and $v$ is connected by a path in $T_i$, we know $p(u)-p(v)\in M_i$. But then  $p(u)-p(v)\in\bigcap_{i=1}^k M_i=0$. Hence $p$ is a constant spline.
\end{proof}

The implications $(3) \Rightarrow (1) \Rightarrow (2)$ hold for any direct product $R=\displaystyle\bigoplus_{i=1}^k R_i$, but Examples \ref{2not1} and \ref{1not3} illustrate how the converse implications can fail if the factors are not irreducible.

\begin{example} \label{2not1}
Let $R=\mathbb{R}[x,y]/(x^2, xy, y^2)$ and consider the graph shown below:
\begin{center}
\scalebox{.9}{\begin{tikzpicture}
%\node at (-1, 1.5);
\draw (0,0) -- (0,3) node[midway, left] {$(x-y)$};
\draw (0,0) -- (2,1.5); %node[midway, right] {$\;\;(\gamma)$}%
\node at (1.25, .4) {$(y)$};
\draw (0,3) -- (2,1.5); %node[midway, right] {$\;\;(\alpha)$};
\node at (1.25, 2.6) {$(x)$};
\draw [fill] (0,0) circle (0.15);
\draw [fill] (0,3) circle (0.15);
\draw [fill] (2,1.5) circle (0.15);
\node [right] at (2,1.5) {$\;0$};
\node [above left] at (0,3) {$x$};
\node [below left] at (0,0) {$y$};
\end{tikzpicture}}
\end{center}
Each cutset of a triangle contains at least two edges, so since the pairwise intersections of the edge labels are trivial, Condition $(2)$ of Theorem \ref{rank1} holds.  But the vertex labeling above is a nonconstant spline, so Condition $(1)$ is false.
\end{example}

\begin{example} \label{1not3}
Consider the labeled graph $(K_4, \alpha)$ over $R = \mathbb{R}[x,y]/(x^2,y^2,xy)$ shown below:  

\begin{center}
\begin{tikzpicture}
\draw (0,0) -- (3,0) node[midway,below] {$(x)$};
\draw  (3,0) -- (3,3) node[midway,right] {$(x+y)$};
\draw (3,3) -- (0,3) node[midway,above] {$(x+y)$};
\draw (0,3) -- (0,0) node[midway,left] {$(y)$};
\draw (0,0) -- (3,3);
\node at (2.5,2) {$(x)$};
\draw (0, 3) -- (3,0);
\node at (.5, 2) {$(y)$};
\draw [fill] (0,0) circle (0.15);
\node at (-.25,-.25) {1};
\draw [fill] (0,3) circle (0.15);
\node at (-.25,3.25) {2};
\draw [fill] (3,0) circle (0.15);
\node at (3.25,-.25) {4};
\draw [fill] (3,3) circle (0.15);
\node at (3.25,3.25) {3};
\end{tikzpicture}
\end{center}

Suppose $p$ is a spline on this graph. Inspect the following triangles:
\begin{itemize}
\item 	triangle $124$ shows $p(4)-p(1)\in(x)\cap(y)=0$
\item triangle $234$ shows $p(4)-p(2)\in(y)\cap(x+y)=0$
\item triangle $314$ shows $p(4)-p(3)\in (x+y)\cap(x)=0$
\end{itemize}
Thus $p(4)=p(i)$ for $i=1,2,3$, so $p$ is a constant spline. 

Hence this graph satisfies Condition $(1)$ in Theorem \ref{rank1}.  However, since $R$ cannot be decomposed as a direct sum (see Remark \ref{remark:irreducible}), to satisfy Condition $(3)$ the graph must have a spanning tree with edges labeled $(0)$, which it does not.
\end{example}

As a corollary to Theorem \ref{rank1}, we characterize when a labeled tree has rank one over {\em any} ring.

\begin{corollary} \label{rank 1 trees}
If $(G,\alpha)$ is a tree over a ring $R$, it is rank one if and only if every edge is labeled (0).
\end{corollary}
\begin{proof}
If every edge is labeled (0), then the graph can be contracted to a single vertex without changing the space of splines, by Lemma \ref{contract zero}.  Hence the space of splines will be generated by the constant splines, and the graph has rank one.

Conversely, since every edge of a tree is a bridge, if there is an edge whose label is {\em not} (0), then by the first part of Theorem \ref{rank1} there is a non-constant spline, so the graph does not have rank one.  (Note that the proof that $(1) \Rightarrow (2)$ in Theorem \ref{rank1} does not impose any restriction on the ring.)
\end{proof}

\subsection{Integers mod $m$.}
As a particular case of Theorem \ref{rank1}, we consider the ring $\Z_m$ of integers modulo $m$.  Splines over $\Z_m$ were previously studied by Bowden and the last author \cite{bt}, who showed that the space of splines on a graph of $n$ vertices could have any rank between 1 and $n$.  We characterize when the space of splines has rank one. 

%\begin{corollary} \label{Zpn rank 1}
%Let $R = \Z_{p^e}$, where $p$ is a prime.  Then a graph $(G, \alpha)$ has rank 1 if and only if $G$ has a spanning tree with all edges labeled $(0)$.
%\end{corollary}

\begin{corollary} \label{Zm rank 1}
Fix an integer $m$ with prime factorization $m = p_1^{e_1}\cdots p_k^{e_k}$ and let $R = \Z_m$.  The graph $(G, \alpha)$ has rank one if and only if $G$ contains spanning trees $T_1, \dots, T_k$ such that all edge labels of $T_i$ are contained in the ideal $(p_i^{e_i})$.  In particular, if $m = p^e$ then $(G, \alpha)$ has rank one if and only if $G$ has a spanning tree with all edges labeled $(0)$.
\end{corollary}

\begin{proof}
	If the prime factorization of $m$ is $m = p_1^{e_1}\cdots p_k^{e_k}$, then the Chinese Remainder Theorem implies
$\Z_m \cong \Z_{p_1^{e_1}} \oplus \cdots \oplus\Z_{p_k^{e_k}}$.  Each ring $\Z_{p_i^{e_i}}$ is uniserial and hence irreducible, and $M_i = (p_i^{e_i})$ for each $i$.  The claim thus follows from Theorem~\ref{rank1} and Remark~\ref{remark:irreducible}.
\end{proof}

\begin{example}
Let $R = \Z_{pq}$, where $p$ and $q$ are distinct primes.  The only ideals in $R$ are $(0), (p), (q), (1)$.  Consider the graph $(K_4, \alpha)$ shown below:

\begin{center}
\begin{tikzpicture}
\draw (0,0) -- (3,0) node[midway,below] {$(p)$};
\draw [ultra thick] (3,0) -- (3,3) node[midway,right] {$(q)$};
\draw (3,3) -- (0,3) node[midway,above] {$(p)$};
\draw [ultra thick] (0,3) -- (0,0) node[midway,left] {$(q)$};
\draw (0,0) -- (3,3);
\node at (2.5,2) {$(p)$};
\draw [ultra thick] (0, 3) -- (3,0);
\node at (.5, 2) {$(q)$};
\draw [fill] (0,0) circle (0.15);
\draw [fill] (0,3) circle (0.15);
\draw [fill] (3,0) circle (0.15);
\draw [fill] (3,3) circle (0.15);
\end{tikzpicture}
\end{center}

The thick and thin edges indicate two spanning trees, one with all edges labeled $(p)$ and the other with all edges labeled $(q)$.  By Corollary \ref{Zm rank 1} this graph has rank one.  Note this graph has rank one even though {\em none} of the edges are labeled $(0)$.
\end{example}

\begin{example}
Instead of simply asking whether a particular labeled graph has rank one, as in the previous example, we could instead ask more generally {\em which} labelings of a given (unlabeled) graph will yield a labeled graph with rank one.  

Motivated by classical splines, we examine the dual graph of a particular triangulation of the plane. We begin with a rectangular grid each of whose squares is divided by a diagonal into two triangles, with the triangles then subdivided by a Clough-Tocher refinement as shown on the left in Figure \ref{F:triangulation}. Starting with an $m \times n$ grid, we denote this graph $G_{m,n}$.  The graph dual to the triangulated grid (ignoring the exterior region) is shown on the right in Figure \ref{F:triangulation}; we denote this graph $G^*_{m,n}$.  

Splines on these graphs were studied by Zhou and Lai \cite{zl}, though we study them over $\Z_r$ rather than the base ring Zhou and Lai used (namely polynomials with two variables).  The module of splines in this example differs very dramatically from that of Zhou and Lai, demonstrating the impact of changing the base ring.

Note that the graph $G^*_{m,n}$ always contains at least two bridges.  If the graph has rank one, then all of the bridges must be labeled $(0)$.   In the following proposition we use Corollary \ref{Zm rank 1} to give a better lower bound on the number of edges labeled $(0)$ than simply by counting bridges. 

\begin{figure}[htbp]
\begin{center}
$$\scalebox{.35}{\includegraphics{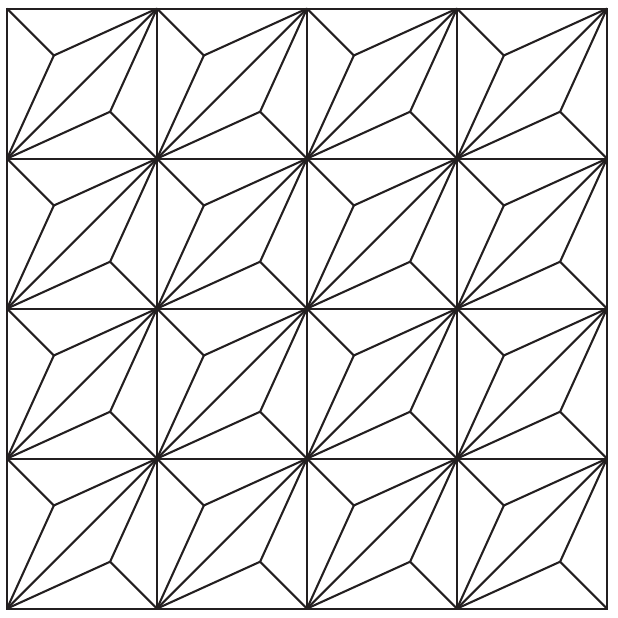}} \qquad \scalebox{.35}{\includegraphics{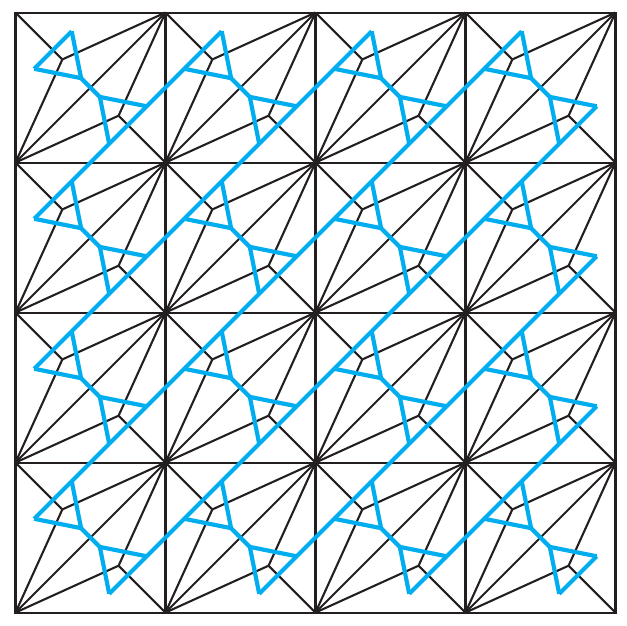}} \qquad \scalebox{.35}{\includegraphics{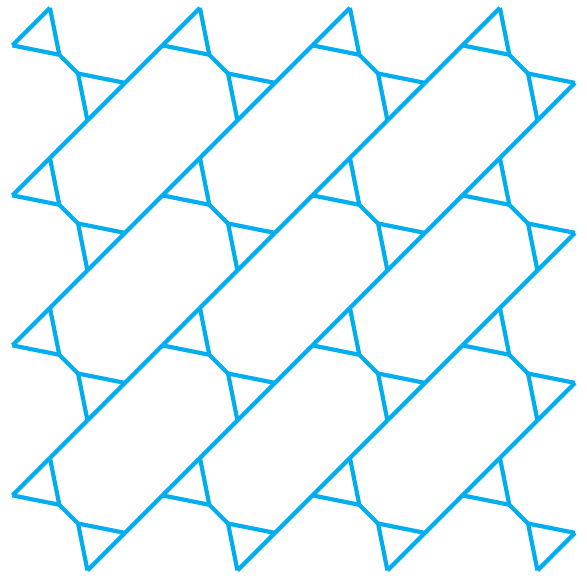}}$$
\end{center}
\caption{Dual graph for a triangulated rectangular grid.}
\label{F:triangulation}
\end{figure}

%We will use Theorem \ref{Zm rank 1} to prove the following fact:

\begin{proposition} \label{triangulated grid}
If $r$ has at most 3 distinct prime factors and $\a$ is a labeling of the edges of $G^*_{m,n}$ over $\Z_r$ such that $(G^*_{m,n}, \a)$ has rank one, then a lower bound on the number of edges labeled $(0)$ is
\begin{itemize}
\item $6mn - 1$ if $r = p^a$,
\item $3mn + m + n - 2$ if $r = p_1^{a_1} p_2^{a_2}$, and
\item $2m + 2n -3$ if $r = p_1^{a_1}p_2^{a_2}p_3^{a_3}$.
\end{itemize}

%some of the edges of the graph must be labeled $(0)$.
\end{proposition}

\begin{proof}
Our argument depends on counting the vertices and edges in $G^*_{m,n}$.  Observe that $G_{m,n}$ has $6mn$ regions and $9mn+m+n$ edges.  The exterior edges do not contribute to the dual graph and there are exactly $9mn-m-n$ interior edges.  So $G^*_{m,n}$ has $6mn$ vertices and $9mn-m-n$ edges.

If $r$ has only one prime factor, then by Corollary \ref{Zm rank 1} the graph $G^*_{m,n}$ has a spanning tree with all edges labeled $(0)$.  Since $G^*_{m,n}$ has $6mn$ vertices, its spanning tree has $6mn-1$ edges.    

Next suppose $\Z_r = \Z_{p^aq^b}$.  Also by Corollary \ref{Zm rank 1}, for every $\a$ the labeled graph $(G^*_{m,n},\a)$ has rank one over $\Z_{p^aq^b}$ if and only if $G^*_{m,n}$ contains spanning trees $T_1$ and $T_2$ with all edge labels of $T_1$ in $(p^a)$ and all edge labels of $T_2$ in $(q^b)$.  These ideals only intersect in $(0)$ so $T_1$ and $T_2$ only overlap on edges labeled $(0)$.

Since $G^*_{m,n}$ has $6mn$ vertices, each spanning tree has $6mn-1$ edges.  Since there are only $9mn-m-n$ edges total, the two spanning trees must overlap in at least $2(6mn-1) - (9mn-m-n) = 3mn +m+n-2$ edges.  Thus for $(G^*_{m,n}, \a)$ to have rank one, there must be at least $3mn+m+n-2$ edges labeled $(0)$.

Finally, suppose $\Z_r = \Z_{p_1^{a_1}p_2^{a_2}p_3^{a_3}}$ where $p_1, p_2, p_3$ are distinct primes.  Then $(G^*_{m,n}, \a)$ has rank one if and only if there are three spanning trees $T_i$ such that all labels of $T_i$ are in $(p_i^{a_i})$.  The intersection $T_1 \cap T_2$ contains at least $3mn+m+n-2$ edges, all with labels contained in $(p_1^{a_1}p_2^{a_2})$.  Hence $(T_1 \cap T_2) \cap T_3$ contains at least 
$$(6mn-1) + (3mn+m+n-2) - (9mn-m-n) = 2m+2n-3 \geq 1$$ 
edges, all of which must be labeled $(0)$. 
\end{proof}

\end{example}

\section{Acknowledgements}
We thank the Institute for Computational and Experimental Research in Mathematics and the American Institute of Mathematics for their generous support of the Research Experience for Undergraduate Faculty (REUF) program, which provided the authors the opportunity to meet for a week during the summers of 2017 and 2018. We also thank the Simons Foundation for its support (\#360097, Alissa Crans) and the National Science Foundation for its support (NSF-DMS 1362855 and 1800773, Julianna Tymoczko). Finally, we are deeply grateful to the referee for their careful reading and helpful suggestions, which greatly improved our work.

%\pagebreak

\bibliographystyle{alpha}
\bibliography{splinesref-1}

%%%%%%%%%%%%%%%%%%%%%%%%%%%%%%%%%%%%%%%%%%%%%%%%%%%%%%%%%%%%%%%%%%%%%%%%%%%%%%%%
%%%%%%%%%%%%%%%%%%%%           END OF FILE                %%%%%%%%%%%%%%%%%%%%%%
%%%%%%%%%%%%%%%%%%%%%%%%%%%%%%%%%%%%%%%%%%%%%%%%%%%%%%%%%%%%%%%%%%%%%%%%%%%%%%%%
\end{document}